\documentclass[11pt,a4paper,reqno]{amsart}
\usepackage{amsmath,amssymb,mathrsfs
}
\usepackage{graphicx,comment}
\usepackage{type1cm}
\usepackage{latexsym}

\usepackage{amsmath, amssymb,ascmac,graphicx}
\calclayout

\allowdisplaybreaks[3]
\numberwithin{equation}{section}
\newtheorem{theorem}{Theorem}
 \newtheorem{remark}{Remark}
\newtheorem{proposition}[theorem]{Proposition}

\newtheorem{example}{Example}

\newtheorem{corollary}[theorem]{Corollary}

\newcommand{\D}{
\mathcal{D}
}
\newcommand{\kk}{\kappa}
\newcommand{\E}{E_\kk}
\newcommand{\dt}{\mathscr{F}_\kappa}

\title{
Hobson's formula for Dunkl operators and its applications
}
\author[N.~Shimeno]{Nobukazu Shimeno}
\address{School of Science and Technology, Kwansei Gakuin University, 
2-1 Gakuen, Sanda, Hyogo 669-1337, Japan}
\email{shimeno@kwansei.ac.jp}
\date{\today}

\begin{document}
\setlength{\abovedisplayskip}{4pt} 
\setlength{\belowdisplayskip}{4pt} 
\maketitle
\thispagestyle{empty}

\begin{abstract}
We generalize classical Hobson's formula concerning partial derivatives of 
radial functions on a Euclidean space to a formula in the Dunkl analysis. As applications we 
give new simple 
proofs of known results involving Maxwell's representation of harmonic 
polynomials, Bochner-Hecke identity, Pizzetti formula for spherical mean, 
and Rodrigues formula for Hermite polynomials. 
\end{abstract}

\section{Introduction}

There exists a nice explicit formula due to Hobson \cite{Hobson1, Hobson2} 
to calculate the action of a constant coefficient linear differential 
operator on radial functions on a Euclidean space, though it is not very well known. 
There are several applications of Hobson's formula such as 
Clebsch projection and Maxwell's representation of harmonic functions, 
Bochner-Hecke identity,  and formulae for  Hermite polynomials 
(\cite{Hobson1, Hobson2, S1,S2}).  
Nomura~\cite{N2} gives a new simple proof of Hobson's formula by using the 
Euler operator. 

In this paper we prove an analogue of Hobson's formula for the Dunkl operators (Theorem~\ref{thm:hobson2}) 
by using a method similar to that of \cite{N2}. 
The Dunkl operators are differential-difference operator on a Euclidean space $\mathbb{R}^d$ 
associated with a finite reflection group $G$, 
which are deformations of directional derivatives. 
For $G=\mathbb{Z}_2^d$, our formula is previously given by Volkmer \cite{volkmer}. 
Moreover, our formula contains the original Hobson's formula as a special case. 
The formula for general $G$ seems to be a new result. 

We give some applications of Hobson's formula for the Dunkl operators analogous 
to those for the Euclidean case mentioned above.  
Though consequences of Hobson's formula presented in this paper are all known 
results, they provide simpler alternative proofs  or  other viewpoints 
in the Dunkl analysis. 
We believe that Hobson's formula is worth more attention in Dunkl analysis 
as well as in 
Euclidean Fourier analysis.


\section{Preliminary results on Dunkl analysis}

In this section we give notation, definitions, and some of known results on Dunkl analysis. 
We refer \cite{DX, R, VK} for details. 

Throughout this paper let $i$ denote the imaginary unit $i=\sqrt{-1}$. 
Let $d$ be a positive integer.  
Let $\langle \,,\,\rangle$ be the standard inner product on $\mathbb{R}^d$ and put $||x||=\langle x, x\rangle^{1/2}$ 
for $x\in\mathbb{R}^d$.  Let 
$R\subset \mathbb{R}^d$ be a reduced root system, which is not necessarily crystallographic. 
For $\alpha\in R$, we write 
$r_\alpha$ for the reflection with respect to the hyperplane $\alpha^\perp$. 
Let $G$ denote the finite reflection group generated by $\{r_\alpha\,:\,\alpha\in R\}$. 
We fix a positive system $R_+\subset R$.

Let $\mathscr{P}=\mathscr{P}(\mathbb{R}^d)$ denote the space of polynomials on $\mathbb{R}^d$. 
For a non-negative integer $m$, let $\mathscr{P}_m$ denote the space of 
polynomials on $\mathbb{R}^d$ that are homogeneous of degree $m$. 

Let $\kk:R\rightarrow \mathbb{R}_{\geq 0},\,\alpha\mapsto \kk_\alpha$ be a $G$-invariant 
function on $R$. We call $\kk$ a (non-negative) multiplicity function. 
We define
\begin{equation}\label{eqn:const}
\gamma_\kk=\textstyle\sum_{\alpha\in R_+}\kk_\alpha, \quad 
\lambda_\kk=\gamma_\kk+(d-2)/2.
\end{equation}

For $\xi\in\mathbb{R}^d$, let 
$\partial_\xi=\langle \xi,\nabla\rangle$ denote the directional derivative corresponding to  $\xi$ and define
the Dunkl operator $\D_\xi=\D_\xi(\kk)$ by
\begin{equation}\label{eqn:dunkl}
\D_\xi f(x)=\partial_\xi f(x)+\sum_{\alpha\in R_+}\kk_\alpha\langle\alpha,\xi\rangle\frac{f(x)-f(r_\alpha x)}
{\langle\alpha,x\rangle}. 
\end{equation}
The Dunkl operators satisfy 
$[\D_\xi,\D_\eta]=0$ for all $\xi,\,\eta\in\mathbb{R}^d$ (\cite{D}). 
Here $[A,B]:=AB-BA$ for operators $A,\,B$. 
Let $\{e_1,\dots,e_d\}$ be the standard orthonormal basis of $\mathbb{R}^d$. 
We write $\partial_j=\partial_{e_j},\,\,\D_j=\D_{e_j}$. The 
Dunkl Laplacian $\Delta_\kk$ is defined by
\begin{equation}
\Delta_\kk=\sum_{j=1}^d \D_{j}^2.
\end{equation}
It has a following expression (\cite{D}):
\begin{equation}\label{eqn:dl}
\Delta_\kk f(x)=\Delta f(x)+\sum_{\alpha\in R_+}
\frac{2k_\alpha}{\langle\alpha,x\rangle}\partial_\alpha f(x)
-\sum_{\alpha\in R_+}\kk_\alpha\frac{\langle\alpha,\alpha\rangle}{\langle\alpha,x\rangle^2}
\{f(x)-f(r_\alpha x)\}.
\end{equation}
The Dunkl operators are homogeneous of degree $-1$ and the Dunkl Laplacian $\Delta_\kk$ is 
homogeneous of degree $-2$. 

Write $\D=(\D_1,\dots,\D_d)$. For $p\in\mathscr{P}$, 
\[
p(\D)=\sum_{g\in G}\D_p^{(g)}g, 
\]
where $\D_p^{(g)}\,\,(g\in G)$ are differential operators uniquely determined by $p$ and $\kk$. 
Let $L_p$ denote the differential operators defined by
\[
L_p=\sum_{g\in G}\D_p^{(g)}\,\,(=p(\D)|_{\text{$G$-inv.poly.}})
\]
For example, $L_p=\partial_j$ for $p(x)=x_j$ and 
\begin{equation}\label{eqn:lap2}
L_{m_2}=\Delta +\sum_{\alpha\in R_+}
\frac{2k_\alpha}{\langle\alpha,x\rangle}\partial_\alpha\,\,(=\Delta_k|_{\text{$G$-inv.poly.}})
\end{equation}
for $m_2(x)=||x||^2=x_1^2+\cdots+x_d^2$. 

For any $y\in\mathbb{R}^d$, 
there exists a unique real analytic function $x\mapsto \E(x,y)=E(x,y)$ 
such that $\D_j \E(\cdot, y)=y_j \E(\cdot,y)\,\,(1\leq j\leq d)$ and $\E(0,y)=1$ (\cite{dj1}). 
We call $\E(x,y)$ the Dunkl kernel. If $\kk=0$, $E_0(x,y)=\exp (\langle x,y\rangle)$. 

Let $h_\kk(x)$ denote the weight function defined by
\begin{equation}\label{eqn:wt}
h_\kappa(x)=\prod_{\alpha\in R_+}|\langle\alpha,x\rangle|^{\kk_\alpha}.
\end{equation}
The Dunkl transform of a function $f\in L^1(\mathbb{R}^d,h_\kk^2(x)\,dx)$ is defined by
\begin{equation}
(\dt f)(y)=b_\kk\int_{\mathbb{R}^d} f(x)\E(x,-iy)h_\kk(x)\,dx, 
\end{equation}
where 
\begin{equation}
b_\kk=\left(\int_{\mathbb{R}^d} h_\kk^2(x)e^{-||x||^2/2}\,dx\right)^{-1}. 
\end{equation}
If $\kk=0$, then $b_0=(2\pi)^{-d/2}$ and $\mathscr{F}_0$ is the Euclidean Fourier transform. 
The Dunkl transform has very nice properties as those of 
the Euclidean Fourier transform. 
For example, 
\begin{equation}\label{eqn:dtmul}
\dt (M_{x_j}f)=i\D_j \dt f\quad (1\leq j\leq d), 
\end{equation}
where $M_p$ denotes the multiplication operator by $p\in\mathscr{P}$ (\cite[Corollary 7.7.1]{DX})．
Moreover, the inversion formula and the Plancherel formula are known for the Dunkl transform (\cite{dj1}).

\section{Hobson's formula}

Now we state and prove an analogue of Hobson's formula for the Dunkl operators. 

\begin{theorem}\label{thm:hobson2}
If $p\in \mathscr{P}_m$, 
$f_0\in C^\infty((0,\infty))$, and $f(x)=f_0(||x||)$, then
\begin{equation}\label{eqn:hobson}
p(\D) f(x)=\sum_{j=0}^{[m/2]}\frac{1}{2^j j!}
\left[\left(\frac{1}{r}\frac{d}{dr}\right)^{m-j}\! f_0\right]\!(||x||)\cdot \Delta_\kk^j \,p(x).
\end{equation}
\end{theorem}

\smallskip
Since $f(x)$ is $G$-invariant, the left hand side of \eqref{eqn:hobson} coincides with $L_pf(x)$. 
If 
$\kk=0$, then Theorem~\ref{thm:hobson2} is the original Hobson's formula (\cite{Hobson1}, \cite[p. 124]{Hobson2}). 
If $G=\mathbb{Z}_2^d$,  then Theorem~\ref{thm:hobson2} is given by Volkmer \cite[Theorem 6]{volkmer}. 

\begin{example}{\rm
If $p(x)=||x||^2$, then $m=2$ and $p(\D)=\Delta_\kk$. By Theorem~\ref{thm:hobson2} we have 
\begin{equation}\label{eqn:laprad}
\Delta_\kk f(x)=L_{m_2} f(x)=f_0''(r)+\frac{d-1+2\gamma_k}{r}f_0'(r)\quad (r:=||x||).
\end{equation}

If we put $f(x)=e^{-||x||^2/2}$ in Theorem~\ref{thm:hobson2}, then we have
\begin{equation}\label{eqn:diffgauss}
p(\D) e^{-||x||^2/2}=\sum_{j=0}^{[m/2]}\frac{(-1)^{m-j}}{2^j j!}
e^{-||x||^2/2} \Delta_\kk^j \,p(x).
\end{equation}
}
\end{example}

\noindent
\textbf{Proof of Theorem~\ref{thm:hobson2}}\,\, 
First notice that the upper bound $[m/2]$ in the summation in \eqref{eqn:hobson} may 
be replaced by $m$. 
We prove the theorem by induction on $m$. 
The case $m=0$ is trivial. Assume that \eqref{eqn:hobson} holds for any element of $\mathscr{P}_m$. 
Let $p\in\mathscr{P}_{m+1}$. Then $E p=(m+1)p$, where $E$ denotes the Euler operator 
\[
E=\displaystyle\sum_{l=1}^d x_l{\partial_l}. 
\]
Thus 
\[
p(\D)=\frac{1}{m+1}\sum_{l=1}^d \D_l\,(\partial_l p)(\D). 
\]
Since $\partial_l p\in\mathscr{P}_m$, it follows from the induction hypothesis that
\begin{equation}\label{eqn:temp1}
p(\D)f(x)=\frac{1}{m+1}\sum_{l=1}^d\sum_{j=0}^{m} \frac{1}{2^j j!}
\D_l\left\{
\left[\left(\frac{1}{r}\frac{d}{dr}\right)^{m-j}f_0\right](||x||)\cdot \Delta_\kk^j \partial_l p(x)
\right\}.
\end{equation}
By the definition of the Dunkl operator \eqref{eqn:dunkl} and the chain rule, we have 
\begin{align*}
\D_l & \left\{
\left[\left(\frac{1}{r}\frac{d}{dr}\right)^{m-j}f_0\right](||x||)\cdot \Delta_\kk^j \partial_l p(x)
\right\}   \\ & = \left[\left(\frac{1}{r}\frac{d}{dr}\right)^{m+1-j}f_0\right](||x||)\cdot x_l \Delta_\kk^j \partial_l p(x)
+\left[\left(\frac{1}{r}\frac{d}{dr}\right)^{m-j}f_0\right](||x||)\cdot \D_l\Delta_\kk^j \partial_l p(x) \\
&= \left[\left(\frac{1}{r}\frac{d}{dr}\right)^{m+1-j}f_0\right](||x||)\cdot\Delta_\kk^j x_l \partial_l p(x) 
-\left[\left(\frac{1}{r}\frac{d}{dr}\right)^{m+1-j}f_0\right](||x||)\cdot 2j \D_l\Delta_\kk^{j-1} \partial_l p(x) \\
& +\left[\left(\frac{1}{r}\frac{d}{dr}\right)^{m-j}f_0\right](||x||)\cdot \D_l\Delta_\kk^j \partial_l p(x).
\end{align*}
The last equality follows from 
\begin{equation}\label{eqn:com00}
[\Delta_\kk^j,M_{x_l}]=2j\D_l\Delta_k^{j-1},
\end{equation}
which is an easy consequence of  
$[\Delta_\kk,M_{x_l}]=2\D_l$ (\cite[Proposition 2.2]{D}, \cite[Lemma~7.1.9]{DX}). 
Substituting the above expression into \eqref{eqn:temp1}, we have
\begin{align*}
p(\D) & =\sum_{j=0}^m \frac{1}{2^j j!} \left[\left(\frac{1}{r}\frac{d}{dr}\right)^{m+1-j}f_0\right](||x||)\cdot
\Delta_\kk^j p(x) \\
& -\sum_{l=1}^d \sum_{j=0}^m 
\frac{1}{2^{j-1}(j-1)!}\left[\left(\frac{1}{r}\frac{d}{dr}\right)^{m+1-j}f_0\right](||x||)\cdot \D_l\Delta_\kk^{j-1} \partial_l p(x) \\
& +\sum_{l=1}^d \sum_{j=0}^m \frac{1}{2^j j!}\left[\left(\frac{1}{r}\frac{d}{dr}\right)^{m-j}f_0\right](||x||)\cdot \D_l\Delta_\kk^j \partial_l p(x) \\
& =\sum_{j=0}^{m+1} \frac{1}{2^j j!} \left[\left(\frac{1}{r}\frac{d}{dr}\right)^{m+1-j}f_0\right](||x||)
\cdot \Delta_\kk^j\, p(x), 
\end{align*}
by using $Ep=(m+1)p$ and changing a summation index. Hence the Theorem is proved. 
\hfill $\square$

\begin{remark}
{\rm 
We imitate the proof of Hobson's formula for $\kk=0$ given by Nomura \cite{N2} in the proof of Theorem~\ref{thm:hobson2}. 
Use of the Euler operator and induction on the degree of $p$ also work well in our case. 
In the case of $\kk=0$,  Nomura proceeds after \eqref{eqn:temp1} by 
using $[\partial_l,\Delta^j]=0$ and $[\Delta^j, E]=2j\Delta^j$. 
Instead of doing in the similar way,  we use \eqref{eqn:com00}, because for general $\kk$ we do not have 
a nice formula for $[\partial_l,\Delta_\kk^j]$. 
}
\end{remark}

\begin{remark}
{\rm Theorem~\ref{thm:hobson2} is closely related with the formula (\cite[Proposition 3.4]{H})
\[
p(\D)=\frac{1}{m!}\left(\text{ad}\,\frac{\Delta_k}{2}\right)^m\! M_p\quad (p\in \mathscr{P}_m), 
\]
which can be proved by using $[\Delta_\kk,M_{x_l}]=2\D_l$ (see \cite[Section 3]{dj2}). 
}
\end{remark}

\section{Applications of Hobson's formula}

In this section, we give some applications of Hobson's formula (Theorem~\ref{thm:hobson2}). 
Though they are all known 
results, Hobson's formula  provides simpler alternative proofs. 

\subsection{$\kk$-harmonic polynomials
}

A polynomial $p$ is called $\kk$-harmonic if $\Delta_\kk \,p=0$. 
Let $\mathscr{H}_{m,\kk}$ denote the space of $\kk$-harmonic polynomials in $\mathscr{P}_m$. 

An explicit formula for the projection operator from $\mathscr{P}_m$ to $\mathscr{H}_{m,\kk}$ 
is derived from Hobson's formula. 

It follows from \eqref{eqn:laprad} by putting $f(x)=||x||^s$ that 
\[
\Delta_\kk ||x||^s=s(s+2\lambda_\kk)||x||^{s-2}.
\]
In particular, $\Delta_\kk ||x||^{-2\lambda_\kk}=0$.

Applying Theorem~\ref{thm:hobson2} to $f(x)=||x||^{-2\lambda_\kk}$, we have
\begin{equation}\label{eqn:maxwell}
p(\D)(||x||^{-2\lambda_\kk}) 
=\sum_{j=0}^{[m/2]}\frac{(-1)^{m-j}2^m (\lambda_\kk)_{m-j}}{2^{2j} j!}
||x||^{-2(\lambda_\kk+m-j)} \Delta_\kk^j \,p(x).
\end{equation}
Thus we have
\begin{equation}\label{eqn:maxwell2}
p(x)=||x||^{2\lambda_\kk+2m}p(\D)(||x||^{-2\lambda_\kk})-\sum_{j=1}^{[m/2]}
\frac{1}{2^{2j}j!(-\lambda_\kk-m+1)_j}||x||^{2j} \Delta_{\kk}^j\, p(x).
\end{equation}
The first term 
$||x||^{2\lambda_\kk+2m}p(\D)(||x||^{-2\lambda_\kk})$ of the right hand side of 
\eqref{eqn:maxwell2} is $\kk$-harmonic by \cite[Theorem 2.3]{X} and the 
remaining term is divisible by $||x||^2$. 
Define the Clebsch projection $\text{proj}_{m,\kk}$ by
\begin{equation}
\text{proj}_{m,\kk}\, p(x)=||x||^{2\lambda_\kk+2m}p(\D)(||x||^{-2\lambda_\kk})\quad (p\in \mathscr{P}_m). 
\end{equation}
Then by \eqref{eqn:maxwell} we have
\begin{equation}\label{eqn:cp}
\text{proj}_{m,\kk}\,p(x)=
\sum_{j=0}^{[m/2]}
\frac{1}{2^{2j}j!(-\lambda_\kk-m+1)_j}||x||^{2j} \Delta_{\kk}^j\, p(x). 
\end{equation}
The formula  
\eqref{eqn:cp} is proved by Dunkl (\cite{D0}, \cite[Theorem 7.1.15]{DX}) using a different method. 

If $p\in\mathscr{H}_{m,\kk}$, then $\text{proj}_{m,\kk} \,p(x)=p(x)$. Hence the mapping 
$\text{proj}_{m,\kk}:\mathscr{P}_m\rightarrow \mathscr{H}_{m,\kk}$ is surjective, 
which gives Maxwell's representation formula for $\kk$-harmonic functions. 
The Clebsch projection 
$\text{proj}_{m,\kk}$ is not injective. Choices of a subset  of $\mathscr{P}_{\kk}$ 
whose image under $\text{proj}_{m,\kk}$  gives a basis of $\mathscr{H}_{m,\kk}$ 
are studied by Xu \cite{X}. 
If $\kk=0$,  formulae for $\text{proj}_{m,0}$ and its surjectivity are classical well-known facts 
(see \cite{Hobson2}, \cite[Vol. I, Ch. VII, \S4, 5]{CH}, \cite[Ch~1, \S 6]{M2}, \cite[Ch. IV \S~2.5]{V}).

\subsection{Bochner-Hecke identity}

As in the Euclidean case (\cite{S1, S2}), an analogue of the Bochner-Hecke identity for the Dunkl transform 
follows from Hobson's formula (Theorem~\ref{thm:hobson2}). 

Let $d\omega$ denote the surface measure on $S^{d-1}$ induced from the measure $dx$ on $\mathbb{R}^d$. 
Define
\begin{equation}
\sigma_{d-1,\kk}=\int_{S^{d-1}}h_\kk(x)\,d\omega(x). 
\end{equation}
Using polar coordinates we have
\begin{equation}
b_\kk^{-1}=2^{\lambda_\kk} \sigma_{d-1,\kk}{\Gamma(\lambda_\kk+1)}.
\end{equation}
By \cite[Proposition~2.8]{D92} (see also \cite[Theorem 1.4.2]{VK}), we have
\begin{equation}\label{eqn:dts}
\int_{S^{d-1}}\E(x,-iy)h_\kk^2(x)\,d\omega(x)=\sigma_{d-1,\kk}\Gamma(\lambda_\kk+1) \left(\frac{||y||}{2}\right)^{-\lambda_\kk}
\! J_{\lambda_\kk}(||y||). 
\end{equation}
Here $J_\nu(r)$ is the Bessel function of the first kind
\begin{equation}
J_\nu(r)=\sum_{j=0}^\infty \frac{(-1)^j (r/2)^{\nu+2j}}{j!\,\Gamma(\nu+j+1)}.
\end{equation}

The following proposition is proved by Gonz\'alez Vieli \cite[Lemma 3.1]{GV}. 
We give a simple proof by using Hobson's formula (Theorem~\ref{thm:hobson2}). 

\begin{proposition}\label{prop:sph1}
For $p\in\mathscr{P}_m$, we have
\[
b_\kk\int_{S^{d-1}}p(x) \E(x,-iy)h_\kk^2(x)\,d\omega(x) =(-i)^m 
\sum_{j=0}^{[m/2]}\frac{(-1)^j}{2^j j!}\frac{J_{\lambda_\kk+m-j}(||y||)}{||y||^{\lambda_\kk+m-j}}\Delta_\kk^j \,p(y).
\]
In particular, if $p\in \mathscr{H}_{m,\kk}$, then
\begin{equation}\label{eqn:sph2}
b_\kk\int_{S^{d-1}}p(x)\E(x,-iy)h_\kk^2(x)\,d\omega(x)= 
\frac{J_{\lambda_\kk+m}(||y||)}{||y||^{\lambda_\kk+m}} p(-iy).
\end{equation}
\end{proposition}
\begin{proof} \,By applying $p(\D)$ with the variable $y$ to \eqref{eqn:dts} and using  
$\D_j \E(x,-iy)=-ix_j\E(x,-iy)$, Theorem~\ref{thm:hobson2} and a formula for the Bessel function \cite[Ch VII, (51)]{E}
\[
\frac{1}{r}\frac{d}{dr}(r^{-\nu}J_\nu(r))=-r^{-\nu-1}J_{\nu+1}(r), 
\]
the proposition follows. 
\end{proof}

The formula 
\eqref{eqn:sph2} is given by \cite[Theorem 3.1]{T}. For 
$\kappa=0$, \eqref{eqn:sph2} is given by \cite[Lemma 2.6.2]{B2}, \cite[Lemma 9.10.2]{AAR}.

We have the following Theorem (\cite[Theorem 4.1]{GV}) as a corollary of 
Proposition~\ref{prop:sph1}. 

\begin{theorem}\label{thm:bochner}
For $p\in\mathscr{P}_m$ and a radial function $f(x)=f_0(||x||)$ we have
\[
\dt (M_p f)(y)
=(-i)^m 
\sum_{j=0}^{[m/2]}\frac{(-1)^j}{2^j j!}(\mathscr{H}_{\lambda_\kk+m-j}f_0)(||y||)\,\Delta_\kk^j \,p(y), 
\]
where 
\[
(\mathscr{H}_\nu f)(s)=\int_0^\infty f(r)\frac{J_\nu(rs)}{(rs)^\nu}r^{2\nu+1}dr
\]
denote the Hankel transform. In particular, if $p\in \mathscr{H}_{m,\kk}$, then 
\begin{equation}\label{eqn:hecke1}
\dt (M_p f)(y)=(-i)^m  (\mathscr{H}_{\lambda_\kk+m}f_0)(||y||)\, p(y).
\end{equation}
\end{theorem}

If $\kappa=0$, Theorem \ref{thm:bochner} is given by \cite[Theorem 1]{S2}. If $\kk=0$, 
\eqref{eqn:hecke1} is given by \cite[Theorem 2]{B1}, \cite[Theorem 2.6.1]{B2}. 
For general $\kappa$, \eqref{eqn:hecke1} is given by \cite[Theorem 3.15]{BS} and \cite[Theorem 3.2]{T}. 

Since $\mathscr{H}_{\lambda_\kk+m}f_0=f_0$ for $f_0 (r)=e^{-r^2/2}$ (\cite[8.6 (10)]{E2}), 
we have an analogue of the  Bochner-Hecke identity from Theorem~\ref{thm:bochner}. 

\begin{corollary}\label{cor:hecke}
For $p\in\mathscr{P}_m$ we have
\[
b_\kk\int_{\mathbb{R}^d}p(x)\,e^{-||x||^2/2}\,h_\kk^2 (x)\,dx
=(-i)^m   \,e^{-||y||^2/2}
\sum_{j=0}^{[m/2]}\frac{(-1)^j}{2^j j!}\,\Delta_\kk^j\, p(y). 
\]
In particular, if $p\in\mathscr{H}_{m,\kk}$, then
\begin{equation}\label{eqn:hecke2}
b_\kk \int_{\mathbb{R}^d}p(x)\,e^{-||x||^2/2}\,h_\kk^2 (x)\,dx
=(-i)^m\,p(y)\,e^{-||y||^2/2}. 
\end{equation}
\end{corollary}

If $\kappa=0$, then \eqref{eqn:hecke2} is the Bochner-Hecke identity 
(cf. \cite[Theorem 2.6.3]{B2}, \cite[Theorem 3.10]{SW}, \cite[Theorem 9.10.3]{AAR}). 
\eqref{eqn:hecke2} for general $\kappa$ is given by \cite[Example 3.16]{BS}.

We have the following corollary from 
\eqref{eqn:diffgauss} and Corollary~\ref{cor:hecke}. 

\begin{corollary}
The following conditions {\rm (1)$\sim$(3)} for 
$p\in\mathscr{P}_m$ are mutually equivalent. \\
{\rm (1)} $p\in\mathscr{H}_{m,\kk}$ 
\\
{\rm (2)} $b_\kk\displaystyle\int_{\mathbb{R}^d}p(x)\,e^{-||x||^2/2}\,h_\kk^2 (x)\,dx
=(-i)^m  \,p(y)\,e^{-||y||^2/2}$
\\
{\rm (3)} $p(\D)e^{-||x||^2/2}=p(-x)e^{-||x||^2/2}$
\end{corollary}

\subsection{Pizzetti's formula}
We have Pizzetti's formula for the Dunkl analysis as a corollary of Proposition~\ref{prop:sph1}. 

\begin{corollary}\label{cor:pizzetti}
For $p\in\mathscr{P}$ we have
\begin{equation}\label{eqn:pizetti}
b_\kk\int_{S^{d-1}}\,p(x)h_\kappa ^2 (x)\,d\omega(x)=
\sum_{l=0}^\infty \frac{(-1)^l\,2^{-\lambda_\kk}}{2^{2l}\,l!\,\Gamma(\lambda_\kappa+l+1)}\,(\Delta_\kappa^l \,p)(0).
\end{equation}
\end{corollary}
\begin{proof}\,First assume that $p\in \mathscr{P}_m$. Substituting $y=0$ to the formula 
in Proposition~\ref{prop:sph1}, we have
\begin{equation}\label{eqn:pizetti01}
b_\kk\int_{S^{d-1}}p(x)h_\kk^2(x)\,d\omega(x)=(-i)^m 
\sum_{j=0}^{[m/2]}\frac{(-1)^j}{2^j j!}\frac{2^{-\lambda_\kk-m+j}}{\Gamma(\lambda_\kk+m-j+1)}(\Delta_\kk^j\,p) (0).
\end{equation}
Notice that $(\Delta_\kk^jp) (0)=0$ for any $j\in\mathbb{Z}_+$, if $m$ is odd. 
Therefore the right hand side of \eqref{eqn:pizetti01} is zero, if $m$ is odd. 
If $m$ is even, say $m=2l$, then \eqref{eqn:pizetti01} becomes
\[
b_\kk \int_{S^{d-1}}p(x)h_\kk^2(x)\,d\omega(x)=
\frac{(-1)^l\,2^{-\lambda_\kk}}{2^{2l} l!\,\Gamma(\lambda_\kk+l+1)}(\Delta_\kk^l\,p) (0). 
\]
The formula 
\eqref{eqn:pizetti} for general $p\in\mathscr{P}$ follows by applying above results for 
each homogeneous component and summing up. 
\end{proof}

Corollary~\ref{cor:pizzetti} is a special case of \cite[Theorem 4.17]{MT}. 
If $\kappa=0$, Corollary~\ref{cor:pizzetti} is a classical formula of Pizzetti (\cite{P}, \cite[Vol II, Ch. IV, \S 3.4]{CH}）. 

If $p\in \mathscr{H}_{m,\kk}$, \eqref{eqn:pizetti} becomes the mean value property 
\[
\sigma_{d-1,\kk}^{-1}\int_{S^{d-1}}p(x)h_\kk^2(x)\,d\omega(x)=p(0). 
\]

\subsection{Hermite polynomials}
For $p\in\mathscr{P}_m$, define the Hermite polynomial $H_{p,\kk}(x)$ by 
\begin{equation}\label{eqn:hermite0}
H_{p,\kk}(x)=\left(-\frac12\right)^m e^{||x||^2}p(\D)e^{-||x||^2}. 
\end{equation}
By Hobson's formula (Theorem~\ref{thm:hobson2}) we have 
\begin{equation}\label{eqn:hermite}
H_{p,\kk}(x)=\sum_{j=0}^{[m/2]}\frac{(-1)^j}{2^{2j}j!}\Delta_\kk^j\,p(x). 
\end{equation}
In particular, if $p\in\mathscr{H}_{m,\kk}$, then $H_{p,\kk}= p$. 

The Dunkl transform of the Hermite function $h_{p,\kk}(x):=e^{-||x||^2/2} H_{p,\kk}(x)$ is given by 
\begin{equation}\label{eqn:dthermite}
\dt h_{p,\kk}=(-i)^m  h_{p,\kk}. 
\end{equation}
Gonz\'alez Vieli \cite[\S 5]{GV} proves \eqref{eqn:dthermite} by using \eqref{eqn:hermite} and Corollary~\ref{cor:hecke}. 

The above results for the Hermite polynomials 
are first given by R\"osler \cite{R0} (with different proof). If $\kk=0$, above formulae are 
proved by Strasburger \cite{S1,S2} by using Hobson's formula. 

\section*{Acknowledgement}
The author thanks Professor Takaaki Nomura and Professor Hiroshi Oda for helpful discussions.

\end{document}